\documentclass[12pt, reqno]{amsart}

\usepackage[a4paper, centering, total={170mm,240mm}]{geometry}

\usepackage{amssymb,latexsym}
\usepackage{blindtext}
\usepackage[backgroundcolor=olive, linecolor=olive, textsize=tiny]{todonotes}
\usepackage{esint,dsfont}

\usepackage{palatino}

\usepackage[colorlinks=true, linkcolor=blue, citecolor=purple, urlcolor=blue]{hyperref}

\def\R{{\mathbb R}}
\def\N{\mathbb{N}}
\def\C{\mathbb{C}}
\def\Z{\mathbb{Z}}

\def\F{\mathbb{F}}

\newtheorem{prop}{\bf Proposition}[section]
\newtheorem{thm}[prop]{\bf Theorem}
\newtheorem{cor}[prop]{\bf Corollary}

\newtheorem{rmk}[prop]{\it Remark}

\newtheorem{ques}[prop]{\bf Question}

\begin{document}

\title[Property $\mathrm{(TTT)}$ and the Cowling--Haagerup constant]{{\bf\Large Property $\mathrm{(TTT)}$ and the Cowling--Haagerup constant}}

\author[I. Vergara]{Ignacio Vergara}
\address{Departamento de Matem\'atica y Ciencia de la Computaci\'on, Universidad de Santiago de Chile, Las Sophoras 173, Estaci\'on Central 9170020, Chile}

\email{ign.vergara.s@gmail.com}
\thanks{This work is supported by the FONDECYT project 3230024.}

\makeatletter
\@namedef{subjclassname@2020}{%
  \textup{2020} Mathematics Subject Classification}
\makeatother

%


\begin{abstract}
We show that Property $\mathrm{(TTT)}$ is an obstruction to weak amenability with Cowling--Haagerup constant $1$. More precisely, if $G$ is a countable group and $H$ is an infinite subgroup of $G$ such that the pair $(G,H)$ has relative Property $\mathrm{(TTT)}$, then the weak Haagerup constant $\boldsymbol\Lambda_{\mathrm{WH}}(G)$ is strictly greater than $1$. We apply this result to some semidirect products and lattices in higher rank algebraic groups.
\end{abstract}


\begingroup
\def\uppercasenonmath#1{} 
\let\MakeUppercase\relax 
\maketitle
\endgroup

\section{{\bf Introduction}}

A countable group $G$ is said to be \emph{amenable} if there is a sequence of finitely supported, positive definite functions $\varphi_n:G\to\C$ such that
\begin{align*}
\lim_{n\to\infty}\varphi_n(s)=1,\quad\forall s\in G.
\end{align*}
More generally, we say that $G$ is \emph{weakly amenable} if there is a sequence of finitely supported functions $\varphi_n:G\to\C$ and a constant $C\geq 1$ such that
\begin{align*}
\lim_{n\to\infty}\varphi_n(s)=1,\quad\forall s\in G,
\end{align*}
and
\begin{align*}
\sup_{n\in\N}\|\varphi_n\|_{B_2(G)}\leq C,
\end{align*}
where $B_2(G)$ stands for the space of Herz--Schur multipliers of $G$; see Section \ref{Sec_WA_WH}. The Cowling--Haagerup constant $\boldsymbol\Lambda(G)$ is the infimum of all $C\geq 1$ such that the condition above holds. This is indeed a weakening of amenability because every positive definite function $\varphi$ satisfies $\|\varphi\|_{B_2(G)}=\varphi(1)$. Hence, if $G$ is amenable, then $\boldsymbol\Lambda(G)=1$. If $G$ is not weakly amenable, we simply set $\boldsymbol\Lambda(G)=\infty$. We refer the reader to \cite{Ver2} for a survey on weak amenability.

A strong negation of amenability is given by Property $\mathrm{(T)}$. Let $G$ be a countable group, and let $H$ be a subgroup of $G$. We say that the pair $(G,H)$ has \emph{relative Property $\mathrm{(T)}$} if every sequence of positive definite functions on $G$ converging pointwise to $1$ converges uniformly on $H$. If this condition holds for $H=G$, we simply say that $G$ has Property $\mathrm{(T)}$. We refer the reader to \cite{BedlHVa} for a thorough account on Property $\mathrm{(T)}$ and its different characterisations. It follows from the definitions that, if $H$ is infinite and the pair $(G,H)$ has relative Property $\mathrm{(T)}$, then $G$ cannot be amenable.

For weak amenability, however, the situation is quite different, and Property $\mathrm{(T)}$ is no longer an obstruction. This was first observed by Cowling and Haagerup \cite{CowHaa}, who showed that the rank 1 Lie groups $\operatorname{Sp}(n,1)$ and $\operatorname{F}_{4,-20}$, as well as their lattices, are weakly amenable. Afterwards, the list of examples was considerably expanded by Ozawa \cite{Oza2}, who showed that all hyperbolic groups are weakly amenable. It must be noted that, in a sense that can be made very precise, ``almost all'' hyperbolic groups have Property $\mathrm{(T)}$; see \cite{Zuk} and \cite{KotKot} for details. In spite of these facts, the following question remains open.

\begin{ques}\label{Ques_T_WH}
Does there exist a pair $(G,H)$ with relative Property $\mathrm{(T)}$ such that $H$ is infinite and $\boldsymbol\Lambda(G)=1$?
\end{ques}

As a matter of fact, all the known examples of groups with $\boldsymbol\Lambda(G)=1$ satisfy the Haagerup property, meaning that there is a sequence of positive definite functions $\varphi_n:G\to\C$ such that $\varphi_n$ vanishes at infinity for all $n$, and
\begin{align*}
\lim_{n\to\infty}\varphi_n(s)=1,\quad\forall s\in G.
\end{align*}
The problem of determining whether this implication holds in general is often referred to as Cowling's conjecture. In order to attack this problem, Knudby \cite{Knu} introduced the weak Haagerup property as a generalisation of both weak amenability and the Haagerup property. In this case, the approximation of the identity is given by functions vanishing at infinity and which are uniformly bounded in $B_2(G)$. We can also associate a constant $\boldsymbol\Lambda_{\mathrm{WH}}(G)$ to this property, which is defined in similar fashion; see Section \ref{Sec_WA_WH}. We always have
\begin{align*}
\boldsymbol\Lambda_{\mathrm{WH}}(G)\leq\boldsymbol\Lambda(G).
\end{align*}

In this paper, we address a variation of Question \ref{Ques_T_WH} by considering a stronger version of Property $\mathrm{(T)}$. In \cite{Oza}, Ozawa introduced Property $\mathrm{(T_P)}$ as a non-equivariant version of Property $\mathrm{(T)}$, where positive definite functions are replaced by positive definite kernels, with some additional assumptions; see Section \ref{Sec_Prop_T} for details. An a priori weaker property, called Property $\mathrm{(TTT)}$, which was also introduced in \cite{Oza}, is defined in terms of boundedness of wq-cocycles; see Section \ref{Sec_Prop_T}. Very recently, Dumas \cite{Dum} showed that Property $\mathrm{(T_P)}$ and Property $\mathrm{(TTT)}$ are equivalent. Using this fact, we prove the following.

\begin{thm}\label{Thm_T_P_WH}
Let $G$ be a countable group, and let $H$ be an infinite subgroup of $G$ such that the pair $(G,H)$ has relative Property $\mathrm{(TTT)}$. Then $\boldsymbol\Lambda_{\mathrm{WH}}(G)$ belongs to $(1,\infty]$.
\end{thm}

Property $\mathrm{(TTT)}$ is strictly stronger than Property $\mathrm{(T)}$, as witnessed by the fact that hyperbolic groups are a-$\mathrm{TTT}$-menable, which is a strong negation of Property $\mathrm{(TTT)}$; see \cite{Oza}. However, these properties are equivalent for some semidirect products; see \cite[Proposition 3]{Oza}. Using this fact, we obtain the following.

\begin{cor}\label{Cor_semidir}
Let $G$ be a countable group such that $G=\Gamma\ltimes A$, where $A$ is infinite and abelian. If the pair $(G,A)$ has relative Property $\mathrm{(T)}$, then $\boldsymbol\Lambda_{\mathrm{WH}}(G)$ belongs to $(1,\infty]$.
\end{cor}

In \cite[Corollary 4.8]{Dum}, Property $\mathrm{(TTT)}$ was established for lattices in higher rank algebraic groups. Hence, by Theorem \ref{Thm_T_P_WH}, such lattices satisfy $\boldsymbol\Lambda_{\mathrm{WH}}(\Gamma)>1$. It turns out that something much stronger is true. The following result is a consequence of \cite{HaaKnu}, \cite{LafdlS} and \cite{Lia}. Since it has not appeared explicitly in the literature, we include it here together with a proof. For details on algebraic groups, we refer the reader to \cite[Chapter 1]{Mar}.

\begin{thm}\label{Thm_alg_grps}
Let $\mathbb{K}$ be a local field, and let $G$ be a connected, almost $\mathbb{K}$-simple $\mathbb{K}$-group with $\operatorname{rank}_{\mathbb{K}}G\geq 2$. Then $G(\mathbb{K})$ does not have the weak Haagerup property. The same holds for any lattice in it.
\end{thm}

This paper is organised as follows. In Section \ref{S_prelim}, we recall the main definitions and properties studied in this manuscript. In Section \ref{S_main}, we prove Theorem \ref{Thm_T_P_WH} and Corollary \ref{Cor_semidir}. Finally, in Section \ref{S_alg_g}, we discuss algebraic groups, and prove Theorem \ref{Thm_alg_grps}.

\subsection*{Acknowledgements}
I'm grateful to Mikael de la Salle for his insightful comments and suggestions.

\section{{\bf Preliminaries}}\label{S_prelim}

\subsection{Schur multipliers}
We begin by reviewing some facts about Schur multipliers and positive definite kernels; for more detailed presentations, we refer the reader to \cite[Chapter 5]{Pis} and \cite[Appendix C]{BedlHVa}. Let $X$ be a set. We say that $\theta:X\times X\to\C$ is a \emph{Schur multiplier} on $X$ if there is a Hilbert space $\mathcal{H}$ and bounded maps $P,Q:X\to\mathcal{H}$ such that
\begin{align}\label{Schur_mult}
\theta(x,y)=\langle P(x),Q(y)\rangle,\quad\forall x,y\in X.
\end{align}
We denote by $V(X)$ the space of Schur multipliers on $X$, and we endow it with the norm
\begin{align*}
\|\theta\|_{V(X)}=\inf\left(\sup_{x\in X}\|P(x)\|\right)\left(\sup_{y\in X}\|Q(y)\|\right),
\end{align*}
where the infimum is taken over all the decompositions as in \eqref{Schur_mult}. In the particular case in which \eqref{Schur_mult} can be obtained with $P=Q$, we say that $\theta$ is a \emph{positive definite kernel} on $X$. In addition, we say that $\theta$ is normalised if $\theta(x,x)=1$ for all $x\in X$. We will make use of the following basic fact; see \cite[Theorem C.3.2]{BedlHVa}.

\begin{prop}\label{Prop_pd_semig}
Let $\xi:X\to\mathcal{H}$ be a map from a set $X$ to a Hilbert space $\mathcal{H}$. For every $t>0$, the map
\begin{align*}
(x,y)\in X\times X\ \longmapsto\ e^{-t\|\xi(x)-\xi(y)\|^2}
\end{align*}
is a normalised, positive definite kernel on $X$.
\end{prop}

\subsection{Weak amenability and the weak Haagerup property}\label{Sec_WA_WH}
Now let $G$ be a locally compact, second countable group. We say that a continuous function $\varphi:G\to\C$ is a Herz--Schur multiplier if the map
\begin{align*}
(x,y)\in G\times G\ \longmapsto\ \varphi(y^{-1}x)
\end{align*}
is a Schur multiplier on $G$. We denote by $B_2(G)$ the space of Herz--Schur multipliers on $G$, and we endow it with the norm inherited from $V(G)$, for which it becomes a Banach algebra. We say that $G$ is weakly amenable if there is a sequence of compactly supported Herz--Schur multipliers $\varphi_n\in B_2(G)$ and a constant $C\geq 1$ such that
\begin{align*}
\lim_{n\to\infty}\sup_{x\in K}|\varphi_n(x)-1|=0,
\end{align*}
for every compact subset $K\subseteq G$, and
\begin{align*}
\sup_{n\in\N}\|\varphi_n\|_{B_2(G)}\leq C.
\end{align*}
The Cowling--Haagerup constant $\boldsymbol\Lambda(G)$ is the infimum of all $C\geq 1$ such that the condition above holds.

More generally, $G$ has the \textit{weak Haagerup property} if there is a sequence $\varphi_n\in B_2(G)$ such that $\varphi_n$ vanishes at infinity for all $n$, and a constant $C\geq 1$ such that
\begin{align*}
\lim_{n\to\infty}\sup_{x\in K}|\varphi_n(x)-1|=0,
\end{align*}
for every compact subset $K\subseteq G$, and
\begin{align*}
\sup_{n\in\N}\|\varphi_n\|_{B_2(G)}\leq C.
\end{align*}
The weak Haagerup constant $\boldsymbol\Lambda_{\mathrm{WH}}(G)$ is the infimum of all $C\geq 1$ such that the condition above holds. It readily follows that $\boldsymbol\Lambda_{\mathrm{WH}}(G)\leq \boldsymbol\Lambda(G)$. We point out that these constants may very well be different. Even more, there are groups $G$ such that $\boldsymbol\Lambda_{\mathrm{WH}}(G)=1$ and $\boldsymbol\Lambda(G)=\infty$. One example of such a group is the wreath product $\Z\wr\F_2$, where $\F_2$ is the free group on 2 generators; see \cite{Knu2} for details.

The weak Haagerup property was introduced in \cite{Knu} with the goal of studying the connection between weak amenability and the Haagerup property. The following result will be essential to our purposes; see \cite[Theorem 1.2]{Knu} and \cite[Proposition 4.3]{Knu}.

\begin{thm}[Knudby]\label{Thm_Knudby}
Let $G$ be a countable discrete group with $\boldsymbol\Lambda_{\mathrm{WH}}(G)=1$. Then there exist a function $\psi:G\to[0,\infty)$, a Hilbert space $\mathcal{H}$, and maps $R,S:G\to\mathcal{H}$ such that
\begin{align}\label{id_Knudby}
\psi(y^{-1}x)=\|R(x)-R(y)\|^2+\|S(x)+S(y)\|^2,\quad\forall x,y\in G,
\end{align}
and $\displaystyle\lim_{x\to\infty}\psi(x)=\infty$.
\end{thm}

Observe that the map $S:G\to\mathcal{H}$ given by Theorem \ref{Thm_Knudby} takes values in the sphere of radius $\frac{1}{2}\sqrt{\psi(1)}$ in $\mathcal{H}$.

\subsection{Property $\mathrm{(T_P)}$ and Property $\mathrm{(TTT)}$}\label{Sec_Prop_T}
Now we discuss two strengthenings of Property $\mathrm{(T)}$ introduced in \cite{Oza}. For simplicity, we shall restrict ourselves to discrete groups. Let $G$ be a group and let $\theta:G\times G\to\C$. For every $s\in G$, we will denote by $s\cdot\theta$ the map
\begin{align*}
s\cdot\theta(x,y)=\theta(s^{-1}x,s^{-1}y),\quad\forall x,y\in G.
\end{align*}
Now let $H$ be a subgroup of $G$. We say that the pair $(G,H)$ has \emph{relative Property $\mathrm{(T_P)}$} if, for every $\varepsilon>0$, there is a finite subset $K\subset G$ and $\delta>0$ such that, for every normalised positive definite kernel $\theta:G\times G\to\C$ satisfying
\begin{align*}
\sup_{s\in G}\|s\cdot\theta-\theta\|_{V(G)}&<\delta & \text{and} & & \sup_{x^{-1}y\in K}|\theta(x,y)-1|<\delta,
\end{align*}
one has
\begin{align*}
\sup_{x,y\in H}|\theta(x,y)-1|<\varepsilon.
\end{align*}

Let $\mathcal{H}$ be a Hilbert space and let $\mathbf{U}(\mathcal{H})$ denote its unitary group. We say that $b:G\to\mathcal{H}$ is a \emph{wq-cocycle} if there is a map $\pi:G\to\mathbf{U}(\mathcal{H})$ (not necessarily a representation) such that
\begin{align*}
\sup_{x,y\in G}\|b(xy)-\pi(x)b(y)-b(x)\| < \infty.
\end{align*}
We say that the pair $(G,H)$ has \emph{relative Property $\mathrm{(TTT)}$} if every wq-cocycle on $G$ is bounded on $H$. If $H=G$, we simply say that $G$ has Property $\mathrm{(TTT)}$. The following was proved in \cite[Theorem 1]{Oza} and \cite[Theorem 2.7]{Dum}.

\begin{thm}[Ozawa, Dumas]\label{Thm_Oza_Dum}
Let $G$ be a countable group, and let $H$ be a subgroup of $G$. Then the pair $(G,H)$ has relative Property $\mathrm{(TTT)}$ if and only if it has relative Property $\mathrm{(T_P)}$.
\end{thm}

\section{{\bf Proof of the main result}}\label{S_main}

Now we prove Theorem \ref{Thm_T_P_WH}. The proof is inspired by that of \cite[Theorem 1]{Oza}.

\begin{proof}[Proof of Theorem \ref{Thm_T_P_WH}]
Let $G$ be a countable group, and let $H$ be an infinite subgroup of $G$ such that $(G,H)$ has relative Property $\mathrm{(TTT)}$. By Theorem \ref{Thm_Oza_Dum}, $(G,H)$ has relative Property $\mathrm{(T_P)}$. Assume by contradiction that $\boldsymbol\Lambda_{\mathrm{WH}}(G)=1$, and let $R,S:G\to\mathcal{H}$ be the maps given by Theorem \ref{Thm_Knudby}. For every $n\geq 1$, we define $\theta_n:G\times G\to(0,1]$ by
\begin{align*}
\theta_n(x,y)=e^{-\frac{1}{n}\|R(x)-R(y)\|^2},\quad\forall x,y\in G.
\end{align*}
By Proposition \ref{Prop_pd_semig}, $\theta_n$ is a normalised positive definite kernel for every $n\geq 1$. Observe that we can view $\mathcal{H}$ as a real Hilbert space with inner product $\mathrm{Re}(\langle\cdot,\cdot\rangle)$. Now consider the full Fock Hilbert space
\begin{align*}
\mathcal{F}=\bigoplus_{n\geq 0}\mathcal{H}^{\otimes n},
\end{align*}
where $\mathcal{H}^{\otimes 0}=\R$ and $\mathcal{H}^{\otimes (n+1)}=\mathcal{H}\otimes\mathcal{H}^{\otimes n}$. The exponential map $\operatorname{EXP}:\mathcal{H}\to\mathcal{F}$ is given by
\begin{align*}
\operatorname{EXP}(\xi)=1 \oplus \frac{\xi}{\sqrt{1!}} \oplus \frac{\xi\otimes\xi}{\sqrt{2!}} 
\oplus \frac{\xi\otimes\xi\otimes\xi}{\sqrt{3!}} \oplus \cdots,\quad\forall\xi\in\mathcal{H}.
\end{align*}
We can define a map from $\mathcal{H}$ to the unit sphere of $\mathcal{F}$ by
\begin{align*}
\operatorname{E}(\xi)=e^{-\|\xi\|^2}\operatorname{EXP}(\sqrt{2}\xi),\quad\forall\xi\in\mathcal{H}.
\end{align*}
This map satisfies
\begin{align*}
\langle\operatorname{E}(\xi),\operatorname{E}(\eta)\rangle=e^{-\|\xi-\eta\|^2},\quad\forall\xi,\eta\in\mathcal{H}.
\end{align*}
For each $n\geq 1$, define $P_n,Q_n:G\to\mathcal{F}$ by
\begin{align*}
P_n(x)&=\operatorname{E}\left(\frac{1}{\sqrt{n}}S(x)\right), & Q_n(x)&=\operatorname{E}\left(-\frac{1}{\sqrt{n}}S(x)\right),
\end{align*}
for all $x\in G$. Recall the identity \eqref{id_Knudby} and observe that, for all $s,x,y\in G$, $n\geq 1$,
\begin{align*}
\theta_n(x,y)-s\cdot\theta_n(x,y) &= e^{-\frac{1}{n}\|R(x)-R(y)\|^2} - e^{-\frac{1}{n}\|R(s^{-1}x)-R(s^{-1}y)\|^2}\\
&= e^{-\frac{1}{n}\psi(y^{-1}x)}\left(e^{\frac{1}{n}\|S(x)+S(y)\|^2} - e^{\frac{1}{n}\|S(s^{-1}x)+S(s^{-1}y)\|^2}\right)\\
&= e^{-\frac{1}{n}\psi(y^{-1}x)}\left(\langle P_n(x),Q_n(y)\rangle - \langle P_n(s^{-1}x),Q_n(s^{-1}y)\rangle \right)\\
&= e^{-\frac{1}{n}\psi(y^{-1}x)}\langle P_n(x)-P_n(s^{-1}x),Q_n(y)\rangle\\
&\quad + e^{-\frac{1}{n}\psi(y^{-1}x)}\langle P_n(s^{-1}x),Q_n(y)-Q_n(s^{-1}y)\rangle.
\end{align*}
Furthermore,
\begin{align*}
\| P_n(x)-P_n(s^{-1}x)\|^2 &=2-2\langle P_n(x), P_n(s^{-1}x)\rangle\\
&= 2-2e^{-\frac{1}{n}\|S(x)-S(s^{-1}x)\|^2}\\
&\leq 2-2e^{-\frac{1}{n}\psi(1)}.
\end{align*}
Similarly, we get
\begin{align*}
\| Q_n(y)-Q_n(s^{-1}y)\|^2 \leq 2-2e^{-\frac{1}{n}\psi(1)}.
\end{align*}
On the other hand, by \cite[Proposition 4.3]{Knu},
\begin{align*}
\left\|(x,y)\mapsto e^{-\frac{1}{n}\psi(y^{-1}x)}\right\|_{V(G)}\leq 1.
\end{align*}
All this shows that
\begin{align*}
\sup_{s\in G}\|s\cdot\theta_n-\theta_n\|_{V(G)}\leq 2\sqrt{2}\left(1-e^{-\frac{1}{n}\psi(1)}\right)^{\frac{1}{2}},
\end{align*}
which tends to $0$ as $n$ goes to infinity. On the other hand, for every finite subset $K\subset G$, we can define
\begin{align*}
M_K=\sup_{x\in K}\psi(x),
\end{align*}
which gives us
\begin{align*}
\sup_{x^{-1}y\in K}\|R(x)-R(y)\|^2\leq M_K.
\end{align*}
Equivalently,
\begin{align*}
\sup_{x^{-1}y\in K}|\theta_n(x,y)-1|\leq 1-e^{-\frac{M_K}{n}},
\end{align*}
which again tends to $0$ as $n$ goes to infinity. Hence, since the pair $(G,H)$ has relative Property $\mathrm{(T_P)}$, there is $n_0\geq 1$ such that
\begin{align*}
\sup_{x,y\in H}|\theta_{n_0}(x,y)-1|<\frac{1}{2}.
\end{align*}
In other words,
\begin{align*}
\sup_{x,y\in H}\|R(x)-R(y)\|^2<n_0\log(2),
\end{align*}
which implies that, for all $x\in H$,
\begin{align*}
\psi(x) &= \|R(x)-R(1)\|^2 + \|S(x)+S(1)\|^2\\
&< n_0\log(2) + \psi(1).
\end{align*}
Since $H$ is infinite and $\psi$ is proper, this is not possible. We conclude that $\boldsymbol\Lambda_{\mathrm{WH}}(G)>1$.
\end{proof}

We can now prove Corollary \ref{Cor_semidir}.

\begin{proof}[Proof of Corollary \ref{Cor_semidir}]
Let $G=\Gamma\ltimes A$, where $A$ is infinite and abelian, and assume that the pair $(G,A)$ has relative Property $\mathrm{(T)}$. By \cite[Proposition 3]{Oza}, $(G,A)$ has relative Property $\mathrm{(TTT)}$. Thus, by Theorem \ref{Thm_T_P_WH}, $\boldsymbol\Lambda_{\mathrm{WH}}(G)>1$.
\end{proof}

\section{{\bf Higher rank algebraic groups}}\label{S_alg_g}

Now we discuss algebraic groups and the proof of Theorem \ref{Thm_alg_grps}. We quickly recall the basic definitions that we will need; for a detailed treatment, we refer the reader to \cite[Chapter 1]{Mar}.

Let $\mathbb{K}$ be a local field, and let $G$ be a $\mathbb{K}$-group. We say that $G$ is almost $\mathbb{K}$-simple if every $\mathbb{K}$-closed normal subgroup of $G$ is finite. Let $\operatorname{rank}_{\mathbb{K}}G$ denote the $\mathbb{K}$-rank of $G$, i.e. the common dimension of maximal $\mathbb{K}$-split tori in $G$. We also denote by $G(\mathbb{K})$ the set of $\mathbb{K}$-rational points of $G$, which is a locally compact group. The following was proved in \cite[Corollary 4.8]{Dum}.

\begin{thm}[Dumas]\label{Thm_Dum_ag}
Let $\mathbb{K}$ be a local field, and let $G$ be a connected, almost $\mathbb{K}$-simple $\mathbb{K}$-group with $\operatorname{rank}_{\mathbb{K}}G\geq 2$. If $\Gamma$ is a lattice in $G(\mathbb{K})$, then $\Gamma$ has Property $\mathrm{(TTT)}$.
\end{thm}

As a consequence, we obtain the following.

\begin{cor}
Let  $\Gamma$ be as in Theorem \ref{Thm_Dum_ag}, then $\boldsymbol\Lambda_{\mathrm{WH}}(\Gamma)>1$.
\end{cor}
\begin{proof}
By Theorem \ref{Thm_Dum_ag}, $\Gamma$ has Property $\mathrm{(TTT)}$. Hence, by Theorem \ref{Thm_T_P_WH}, $\boldsymbol\Lambda_{\mathrm{WH}}(\Gamma)>1$.
\end{proof}

As mentioned in the introduction, we actually have $\boldsymbol\Lambda_{\mathrm{WH}}(\Gamma)=\infty$ for such lattices. This is a consequence of a series of different results that we record here for the reader's convenience.

In \cite{Haa}, Haagerup showed that higher rank simple Lie groups with finite centre are not weakly amenable. The proof focuses on the groups $\operatorname{SL}_3(\R)$ and $\operatorname{Sp}_4(\R)$, and the general result follows from these two particular cases. This strategy has been adapted to other similar contexts; see \cite{HaaKnu, HaadLa, Lia, Ver}. The proof of Theorem \ref{Thm_alg_grps} follows the same ideas.

Let $\mathbb{K}$ be a field. The special linear group $\operatorname{SL}_3(\mathbb{K})$ is the group of $3\times 3$ matrices with coefficients in $\mathbb{K}$ and determinant 1. The symplectic group $\operatorname{Sp}_4(\mathbb{K})$ is given by all matrices $A\in\operatorname{GL}_4(\mathbb{K})$ such that $A^tJA=J$, where
\begin{align*}
J=\left(\begin{array}{cccc}
0 & 0 & 1 & 0\\
0 & 0 & 0 & 1\\
-1 & 0 & 0 & 0\\
0 & -1 & 0 & 0\\
\end{array}\right).
\end{align*}

The following was proved in \cite[Proposition 4.1]{LafdlS} and \cite[Proposition 5.1]{LafdlS}.

\begin{thm}[Lafforgue--de la Salle]\label{Thm_LafdlS}
Let $\mathbb{K}$ be either $\R$ or a nonarchimedian local field. Then $\boldsymbol\Lambda_{\mathrm{WH}}(\operatorname{SL}_3(\mathbb{K}))=\infty$.
\end{thm}

\begin{rmk}
The results in \cite{LafdlS} are expressed in a different (more general) language. They involve the spaces $MS^p(L^2(G))$ of multipliers of Schatten classes. Here we are only interested in the case $p=\infty$, for which $\|\check{\varphi}\|_{MS^\infty(L^2(G))}=\|\varphi\|_{B_2(G)}$; see \cite[\S 1]{LafdlS} for details.
\end{rmk}

In the real case, this result was extended in \cite{HaadLa} to the symplectic group $\operatorname{Sp}_4(\R)$ (where it is called $\operatorname{Sp}(2,\R)$). Using this, the following was obtained in \cite[Theorem B]{HaaKnu}.

\begin{thm}[Haagerup--Knudby]\label{Thm_HaaKnu}
Let $G$ be a connected simple Lie group with $\operatorname{rank}_{\mathbb{R}}G\geq 2$. Then $\boldsymbol\Lambda_{\mathrm{WH}}(G)=\infty$.
\end{thm}

In the nonarchimedian case, a similar result for $\operatorname{Sp}_4(\mathbb{K})$ was obtained in \cite[Theorem 2.4]{Lia}. Again, the language is quite different, but a particular case of that result can be stated as follows.

\begin{thm}[Liao]\label{Thm_Liao}
Let $\mathbb{K}$ be a nonarchimedian local field. Then $\boldsymbol\Lambda_{\mathrm{WH}}(\operatorname{Sp}_4(\mathbb{K}))=\infty$.
\end{thm}

All these results together allow us to obtain the same conclusion for more general higher rank algebraic groups.

\begin{proof}[Proof of Theorem \ref{Thm_alg_grps}]
Let $G$ be a connected, almost $\mathbb{K}$-simple $\mathbb{K}$-group with $\operatorname{rank}_{\mathbb{K}}G\geq 2$. By \cite[Proposition 1.6.2]{Mar} and \cite[Proposition 2.3.4]{Mar}, $G(\mathbb{K})$ has a closed subgroup $F$ which is isomorphic to a quotient of either $\operatorname{SL}_3(\mathbb{K})$ or $\operatorname{Sp}_4(\mathbb{K})$ by a finite subgroup. Hence, by \cite[Theorem A.(1)]{Knu2},
\begin{align*}
\boldsymbol\Lambda_{\mathrm{WH}}(G(\mathbb{K}))\geq \boldsymbol\Lambda_{\mathrm{WH}}(F).
\end{align*}
Moreover, by \cite[Theorem A.(2)]{Knu2}, together with Theorems \ref{Thm_LafdlS}, \ref{Thm_HaaKnu} and \ref{Thm_Liao}, we get
\begin{align*}
\boldsymbol\Lambda_{\mathrm{WH}}(F)=\infty.
\end{align*}
This shows that $G(\mathbb{K})$ does not have the weak Haagerup property. By \cite[Theorem A.(6)]{Knu2}, the same holds for any lattice in $G(\mathbb{K})$.
\end{proof}

\bibliographystyle{plain} 

\bibliography{Bibliography}

\end{document}